\documentclass[12pt,letterpaper,reqno]{amsart}

\usepackage{times}
\usepackage[T1]{fontenc}
\usepackage{mathrsfs}
\usepackage{latexsym}
\usepackage[dvips]{graphics}
\usepackage{epsfig}
\usepackage{amsmath,amsfonts,amsthm,amssymb,amscd}
\input amssym.def
\input amssym.tex

\newcommand\be{\begin{equation}}
\newcommand\ee{\end{equation}}
\newcommand\bea{\begin{eqnarray}}
\newcommand\eea{\end{eqnarray}}
\newcommand\bi{\begin{itemize}}
\newcommand\ei{\end{itemize}}
\newcommand\ben{\begin{enumerate}}
\newcommand\een{\end{enumerate}}
\newcommand\bc{\begin{center}}
\newcommand\ec{\end{center}}
\newcommand\ba{\begin{array}}
\newcommand\ea{\end{array}}

\newtheorem{thm}{Theorem}[section]

\newtheorem{lem}[thm]{Lemma}

\newtheorem{defi}[thm]{Definition}

\theoremstyle{definition}
\newtheorem{rek}[thm]{Remark}

\begin{document}

\title{The inverse problem for representation functions for general linear 
forms}

\author{Peter Hegarty}
\email{hegarty@math.chalmers.se} \address{Mathematical Sciences,
Chalmers University Of Technology and G\"oteborg University,
G\"oteborg, Sweden}

\subjclass[2000]{11B13, 11B34, 11B05} \keywords{Integer basis, representation
function, linear form}

\date{\today}

\begin{abstract} The inverse problem for representation functions takes as 
input a triple $(\mathbb{X},f,\mathscr{L})$, where $\mathbb{X}$ is 
a countable 
semigroup, $f : \mathbb{X} \rightarrow \mathbb{N}_{0} \cup \{\infty\}$
a function, $\mathscr{L} : a_{1}x_{1} + \cdots + a_{h} x_{h}$ an 
$\mathbb{X}$-linear form and asks for a subset $A \subseteq \mathbb{X}$ 
such that there are $f(x)$ solutions (counted appropriately) 
to $\mathscr{L}(x_1,...,x_h) = x$ for
every $x \in \mathbb{X}$, or a proof that no such subset exists. 
\par This paper
represents the first systematic study of this problem for arbitrary linear
forms when $\mathbb{X} = \mathbb{Z}$, the setting which in many respects is the
most natural one. Having first settled on the $\lq$right' way to count 
representations, we prove that every primitive form has a unique representation
basis, i.e.: a set $A$ which represents the function $f \equiv 1$. We
also prove that a partition regular form (i.e.: one for which no 
non-empty subset of the coefficients sums to zero) represents any function $f$
for which $\{f^{-1}(0)\}$ has zero asymptotic density. These
two results answer questions recently posed by Nathanson. 
\par The inverse
problem for partition irregular forms seems to be more 
complicated. The simplest example of such a form is $x_1 - x_2$, and for this
form we provide some partial results. Several remaining open problems are
discussed.  
  
\end{abstract}


\maketitle

\setcounter{equation}{0}

\setcounter{equation}{0}

\section{Introduction and Definitions}

A fundamental notion in additive number theory is that of {\em basis}. 
Given a positive integer $h$, a subset $A \subseteq \mathbb{N}_0$ for which
$0 \in A$ is said to be a
{\em basis} for $\mathbb{N}_0$ of 
{\em order} $h$ if, for every $n \in \mathbb{N}_0$ the equation
\be
x_1 + x_2 + \cdots + x_h = n
\ee
has at least one solution in $A$. The requirement that $0 \in A$ means that, in
words, $A$ is a basis of order $h$ if every positive integer can be written
as the sum of at most $h$ positive integers from $A$. 
\par In classical number theory, we encounter
questions of the type : is the following set $A$ 
a basis for $\mathbb{N}_0$ and, if 
so, of what order ? Famous examples include the cases when $A$ is the 
set $\mathbb{N}_{0}^{k}$ of perfect $k$:th powers, for some fixed $k$ 
(Waring's Problem), or the set $\mathbb{P}_{0,1}$ of
primes together with 0 and 1 (Goldbach's Problem). 
In both these cases, it is in fact more natural
to consider a slightly weaker notion, namely that of {\em asymptotic basis}.
A subset $A \subseteq \mathbb{N}_0$ is said to be an asymptotic basis of order
$h$ if (1.1) has a solution for every $n \gg 0$. For example, in Waring's
Problem, if $g(k)$ and $G(k)$ denote the order, resp. asymptotic order, of
the set $\mathbb{N}_{0}^{k}$, then it is known that $G(k)$ is considerably less
than $g(k)$ for large $k$. Regarding the primes, Vinogradov's Theorem says that
$\mathbb{P}_{0,1}$ 
is an asymptotic basis of order 4, while it remains open as to 
whether it is actually a basis of even that order. Goldbach's conjecture
would imply the much stronger result that $\mathbb{P}_{0,1}$ is a basis of 
order 3. In this regard, it is well-known that the subset of the positive
even integers representable as the sum of two primes has 
asymptotic density one (see, for example, \cite{Va} Theorem 3.7). This 
motivates a further fairly natural weakening of the notion of basis. In
the terminology of \cite{Na4}, we say that $A \subseteq \mathbb{N}_0$ is a 
{\em basis of order $h$ for almost all $\mathbb{N}_0$} if the set of
$n \in \mathbb{N}_0$ for which (1.1) has a solution in $A$ has
asymptotic density one. 
\par In the terminology commonly used by practitioners of the subject, the 
above classical problems are illustrations of a {\em direct problem}, where 
we are in essence 
seeking a description of the $h$-fold sumset of a specified set $A$. 
The corresponding {\em inverse problem} is to construct a set
$A$ with a specified so-called {\em (unordered) 
representation function} of a certain 
order. Let $f : \mathbb{N}_0 \rightarrow \mathbb{N}_0$ be any function, 
$h \in \mathbb{N}$ and $A \subseteq \mathbb{N}_0$. We say that $f$ is the
(unordered) representation function of $A$ of order $h$ if, for every 
$n \in \mathbb{N}_0$, 
\be
f(n) = \# \{(x_1,...,x_h) \in A^h : x_1 \leq x_2 \leq \cdots \leq x_h \; 
{\hbox{and (1.1) holds}} \}.
\ee  
If (1.2) holds then we write $f = f_{A,h}$. The relationship between bases
and representations functions is thus that 
\\
\par - $A$ is a basis of order $h$ if and only if $f_{A,h}^{-1} (0)$ is empty,
\par - $A$ is an asymptotic basis of order $h$ if and only if 
$f_{A,h}^{-1} (0)$ is finite, 
\par - $A$ is a basis of order $h$ for almost all $\mathbb{N}_0$ if and
only if $d[f_{A,h}^{-1}(0)] = 0$. 
\\
\\
The inverse problem for bases/representation functions 
in $\mathbb{N}_0$ is, in general, very
hard. Probably the single most famous illustration of this is the long-standing
question of Erd\H{o}s and Tur\'{a}n \cite{ETu} as to whether there exists an
asymptotic basis of any order $h$ whose representation function is 
bounded. Not much is known beyond the facts that, on the one hand,
$f_{A,2}$ cannot be ultimately constant \cite{D} while, on the other, there
exist for every $h$ so-called {\em thin} bases $A_h$ satisfying 
$f_{A_h,h} (n) = \Theta (\log n)$ \cite{ETe}.
\par In seeking a more tractable inverse problem, a natural starting point are
the following two observations :
\\
\\
{\bf I}. 
The various notions of basis make sense in any additive semigroup, not just
$\mathbb{N}_0$.
\\
{\bf II}. Intuitively, it is easy to see why the inverse problem is hard in
$\mathbb{N}_0$. Namely, when trying to construct a set $A$ with a given
representation function $f$ of a given order $h$, 
we cannot use negative numbers to help $\lq$fill 
in gaps'. More precisely, suppose we try to construct our set $A$ one element 
at a time and at some point have constructed a finite set $A^{\prime}$ such 
that 
\be
f_{A^{\prime},h}(n) \leq f(n) \; {\hbox{for every $n \in \mathbb{N}_0$}}. 
\ee
Assuming
$f^{-1}(0)$ is finite, say, there will
be a smallest $n = n_1$ for which we have strict inequality in (1.3). 
We would now like
to add some more elements to $A^{\prime}$ which create a new solution to (1.1) 
for $n=n_1$ while not violating (1.3). If we could use negative numbers then,
as long as $f^{-1}(0)$ is finite,
a natural way to do this would be to add to $A^{\prime}$ exactly $h$ new elements
which (a) don't all have the same sign (b) are all much larger in absolute
value than anything currently in $A^{\prime}$ (c) almost cancel each other
out in exactly one way, in which case they add up to $n_1$. 
\\
\\
These observations led Nathanson to consider the inverse problem for
representation functions in $\mathbb{Z}$ or, more generally, in countable 
abelian groups. The fundamental result showing that
we have a much more tractable problem in this setting is the 
following :

\begin{thm}\label{thm:zbases} \cite{Na2}
Let $f : \mathbb{Z} \rightarrow \mathbb{N}_0 \cup \{\infty\}$ be any 
function for which $f^{-1}(0)$ is a finite set. Then for every $h \in 
\mathbb{N}_{\geq 2}$ there exists a subset $A \subseteq \mathbb{Z}$ such
that $f_{A,h} = f$. 
\end{thm}

In particular, the Erd\H{o}s-Tur\'{a}n question has a positive answer in
$\mathbb{Z}$ : we can even construct a set $A$ such that $f_{A,h}(n) = 1$
for every $n$, a so-called {\em unique representation basis of order $h$}
for $\mathbb{Z}$. Nathanson's proof of Theorem \ref{thm:zbases} follows
the idea in observation {\bf II} above. 
\\
\\
Now that we have a more tractable problem, we can look to push our
investigations deeper. One line of enquiry which seems natural is to extend
the basic notion of basis further by replacing the left-hand side of (1.1) by 
an arbitrary linear form $a_1 x_1 + \cdots + a_h x_h$. If 
the $a_i$ are assumed to integers, then this idea makes
sense in any additive semigroup, otherwise one should work in a 
commutative ring. For the remainder of this paper, though, we shall 
always be working in $\mathbb{Z}$, but the interested reader is invited to
extend the discussion to a more general setting. Note that the various notions
of basis are only meaningful if the linear form is {\em primitive}, i.e.:
if the coefficients are relatively prime. This will be assumed throughout.
\par We now start with a couple of formal definitions. 

\begin{defi}\label{defi:linbasis}  
Let $a_1,...,a_h$ be relatively prime non-zero integers and let 
$\mathscr{L} = \mathscr{L}_{a_1,...,a_h}$ denote the linear form 
$a_1 x_1 + \cdots + a_h x_h$. A subset $A \subseteq \mathbb{Z}$ is said to be
an {\em $\mathscr{L}$-basis} if the equation
\be
a_1 x_1 + \cdots + a_h x_h = n
\ee
has at least one solution for every $n \in \mathbb{Z}$.
\par Similarly, we say that $A$ is an {\em asymptotic 
$\mathscr{L}$-basis} if (1.4) has a solution for all but finitely many $n$, 
and that $A$ is an {\em $\mathscr{L}$-basis for almost all $\mathbb{Z}$} if
those $n$ for which (1.4) has no solution form a set of asymptotic 
density zero.   
\end{defi}

\begin{rek}
Recall that a subset $S \subseteq \mathbb{Z}$ is said to have asymptotic 
density zero if 
\be
\lim_{n \rightarrow + \infty} {|S \cap [-n,n]| \over 2n+1} = 0.
\ee
\end{rek}
 
To generalise the notion of unordered representation function to 
arbitrary linear forms requires a bit more care. The definition we give below 
is, we think, the natural one. First we need some terminology. A solution
$(x_1,...,x_h)$ of (1.4) is said to be a {\em representation} of $n$ by the
form $\mathscr{L} = \mathscr{L}_{a_1,...,a_h}$. We say that two
representations $(x_1,...,x_h)$ and $(y_1,...,y_h)$ of the same integer 
$n$ are {\em equivalent} if, for every $\xi \in \mathbb{Z}$, 
\be
\sum_{x_i = \xi} a_i = \sum_{y_i = \xi} a_i.
\ee
We now define the {\em (unordered) $\mathscr{L}$-representation function} 
$f_{A,\mathscr{L}}$ of
a subset $A \subseteq \mathbb{Z}$ as 
\be
f_{A,\mathscr{L}}(n) = \# \{{\hbox{equivalence classes of representations of
$n$ by $\mathscr{L}$}} \}.
\ee
There are a few existing 
results on the inverse problem for bases for general linear forms. Indeed in
\cite{SS} the problem was already raised in the much more difficult setting
of $\mathbb{N}_0$, in which case one may assume that the coefficients $a_i$ 
in (1.4) are positive. No results were proven in that paper, 
and none of the specific 
problems the authors posed have, to the best of our knowledge, been 
settled since. They do make the intriguing observation, though, that
for some forms one can construct a unique representation basis for 
$\mathbb{N}_0$, for example the form $x_1 + ax_2$ for any $a > 1$. It 
would be fascinating to have a full classification of the forms for which
this is possible. The one result of note we are aware of is Vu's extension 
\cite{Vu} of the Erd\H{o}s-Tetali result on thin bases to general linear forms.
\par In the setting of $\mathbb{Z}$, one is first and foremost interested in 
generalising Theorem \ref{thm:zbases}. There are some recent results of 
Nathanson \cite{Na3} on binary forms, and in \cite{Na4} he poses some 
problems for general forms. Our results answer some of his questions and 
supersede those in \cite{Na2}. 
\par It should be noted here that in all the papers referenced above, only
the ordered representation function is considered, meaning that one 
distinguishes between equivalent representations of the same
number. For Vu's result, this distinction is not important (since his is a
$\Theta$-result), but the results we shall prove here have a much more 
elegant formulation when one works with unordered representations. 
\\
\\
We close this section by briefly summarising the results to follow. In 
Section 2 we prove that for any primitive form $\mathscr{L}$ there 
exists a unique representation basis. This generalises the main result of
\cite{Na1} and answers Problem 16 of \cite{Na4}. Our method is founded
on observation {\bf II} on page 2 and is thus basically the same as that
employed in these earlier papers. However, we believe our presentation
is much more streamlined, especially when specialised to the forms 
$x_1 + \cdots + x_h$. 
\par In Section 3 we seek a generalisation of 
Theorem \ref{thm:zbases}. We introduce the notion of an {\em automorphism}
of a linear form and show that a form has no
non-trivial automorphisms (we will say what $\lq$non-trivial' means) if and
only if it is {\em partition regular} in the sense of Rado, i.e.: no non-empty
subset of the coefficients sums to zero. Our main result in this section 
is that, 
if $\mathscr{L}$ is partition regular then,  
for any $f : \mathbb{Z} \rightarrow \mathbb{N}_0 \cup \{\infty\}$ such that
the set $f^{-1}(0)$ has density zero, there exists $A \subseteq 
\mathbb{Z}$ for which $f_{A,\mathscr{L}} = f$. Since any form all of whose
coefficients have the same sign is partition regular, this 
result generalises Theorem \ref{thm:zbases}. But it also extends that
theorem, since we 
only require $f^{-1}(0)$ to have density zero, and not necessarily be 
finite. It thus answers 
Problem 13 and partly resolves Problem 17 in \cite{Na4}. 
\par Irregular forms seem to be harder to deal
with. The simplest such form is $\mathscr{L}_{1,-1} : x_1 - x_2$. In Section 4
we study this form but our results are weaker than those in Section 3. Open
problems remain and these are discussed in Section 5.
\par Finally, note that the methods of proof in Sections 3 and 4 are
in essence no different from those in Section 2. The main point here is
in identifying the $\lq$right' theorems, but once 
this is done no really new ideas are needed to carry out the proofs.       
 
\setcounter{equation}{0}
\section{Unique representation bases}\label{sec:unique}
Before stating and proving the main result of this section, we introduce some 
more notation and terminology similar to that in (1.6)
above. Let $\mathscr{L} : a_1 x_1 + \cdots + a_h x_h$ be a linear form. 
Let $m,p$ be positive integers, $(r_1,...,r_m)$ any $m$-tuple and
$(s_1,...,s_p)$ any $p$-tuple of integers, and 
\be
\pi_{1} : \{1,...,m\} \rightarrow \{1,...,h\}, \;\;\;
\pi_{2} : \{1,...,p\} \rightarrow \{1,...,h\},
\ee
any functions. 
We say that the sums $\sum_{i=1}^{m} a_{\pi_{1}(i)} r_{i}$ and
$\sum_{i=1}^{p} a_{\pi_{2}(i)} s_{i}$ are {\em equivalent w.r.t. $\mathscr{L}$},
and write
\be
\sum_{i=1}^{m} a_{\pi_{1}(i)} r_{i} \equiv_{\mathscr{L}}
\sum_{i=1}^{p} a_{\pi_{2}(i)} s_{i}
\ee
if, for every $\xi \in \mathbb{Z}$,
\be
\sum_{r_{i} = \xi} a_{\pi_{1}(i)} = \sum_{s_{i} = \xi} a_{\pi_{2}(i)}.
\ee
Note that this generalises the notion of equivalent representations in 
Section 1 since the latter corresponds to the 
special case $m=p=h$, $\pi_1 = \pi_2 = {\hbox{id}}$.
\\
\\
The remainder of this section is devoted to the proof of the following 
result :

\begin{thm}\label{thm:unique} (i) Let $\mathscr{L}$ be a linear form. 
Then there exists a
unique representation $\mathscr{L}$-basis if and only if 
$\mathscr{L}$ is primitive.
\\
(ii) Let $\mathscr{L} = \mathscr{L}_{a_{1},...,a_{h}}$. The following are
equivalent :
\par (a) $\mathscr{L}$ is primitive and not all $a_{i}$ have the same sign,
\par (b) for every $n \in \mathbb{Z}$, there exists a unique
represenntation 
$\mathscr{L}$-basis $A(n) \subseteq [n,+ \infty)$.
\end{thm}

\begin{proof}
We concentrate on proving part (i) :  the proof of part (ii) will then be 
an immediate 
consequence of our approach. 
Clearly there no $\mathscr{L}$-basis if $\mathscr{L}$ is
imprimitive, so suppose $\mathscr{L}$ is primitive.
The result is trivial if $\mathscr{L} = \mathscr{L}_{\pm 1}$, 
so we may assume that $\mathscr{L}$ is a function of at least two
variables. We find it
convenient to use the slightly unusual 
notation $\mathscr{L} = \mathscr{L}_{a_{1},...,a_{h+1}}$,
where
$h \geq 1$. Henceforth, we deal with a fixed form $\mathscr{L}$, so $h$ and the
coefficients $a_{i}$ are fixed. Our task is to construct a unique
representation basis
$A$ for $\mathscr{L}$. Let $d_{0}$ be any non-zero integer and put $A_{0} :=
\{d_{0}\}$. We will construct the set $A$ step-by-step as
\be
A = \bigsqcup_{k=0}^{\infty} A_{k},
\ee
where, for each $k>0$, the set $A_{k}$ will consist of $h+1$ suitably chosen
integers, which we denote as
\be
A_{k} = \{d_{k,1},...,d_{k,h},e_{k}\}.
\ee
We adopt the following ordering of the integers :
\be
0,1,-1,2,-2,3,-3,...,
\ee
and denote the ordering by $\mathscr{O}$. For each $k >0$ the elements of
$A_{k}$ will be chosen so that
\par {\bf (I)} $A_{k}$ represents the least integer $t_{k}$ in the
ordering $\mathscr{O}$ not already represented by $B_{k-1} := \sqcup_{j
=0}^{k-1} A_{j}$,
\par {\bf (II)} no integer is represented more than once by $B_k$.
\\
\\
Since the set $B_{0}=A_{0}$ clearly already satisfies property {\bf (II)},
it is then clear that if both {\bf (I)} and {\bf (II)} are satisfied for
every $k>0$, then the set $A$ given by (2.5) will be a unique 
representation basis.
\\
\\
Since $\mathscr{L}$ is primitive, it represents 1. Fix a choice
$(s_{1},...,s_{h+1})$ of a representation of 1. Let $M$ be a fixed, very
large positive real number (how large $M$ needs to be will become clear in
what follows).
\\
\\
Fix $k > 0$. Suppose $A_{0},...,A_{k-1}$ have already been chosen in order
to satisfy {\bf (I)} and {\bf (II)}.
Let $t_{k}$ be the least integer in the ordering
$\mathscr{O}$ not represented by $B_{k-1}$. First choose any $h$ positive
numbers $\delta_{k,1},...,\delta_{k,h}$ such that
\be
{\delta_{k,1} \over d_{k-1,h}} > M, \;\;\; {\delta_{k,i+1} \over
\delta_{k,i}} > M, \;\;\; {\hbox{for $i=1,...,h-1$,}}
\ee
and put
\be
\epsilon_{k} := - \lfloor {1 \over a_{h+1}} \left( \sum_{i=1}^{h} a_{i}
\delta_{k,i} \right) \rfloor.
\ee
Let
\be
\sum_{i=1}^{h} a_{i}\delta_{k,i} + a_{h+1}e_{k} := u_{k} \in [0,a_{h+1}],
\ee
and choose the elements of $A_{k}$ as
\be
(d_{k,1},...,d_{k,h},e_{k}) :=
(\delta_{k,1},...,\delta_{k,h},\epsilon_{k}) + (t_{k}-u_{k}) \cdot
(s_{1},...,s_{h+1}).
\ee
Our choice immediately guarantees that {\bf (I)} is satisfied. The
remainder of the proof  
is concerned with showing that {\bf (II)} still holds
provided the integer $M$ is sufficiently large. This is done by
establishing the following two claims :
\\
\\
{\sc Claim 1} : {\em Let $(x_{1},...,x_{h+1}), (y_{1},...,y_{h+1})$ be any
two $(h+1)$-tuples of integers in $B_{k}$. Then exactly one of the
following holds :
\par (i) \be
\sum_{x_{i} \in A_{k}} a_{i} x_{i} \equiv_{\mathscr{L}}
\sum_{y_{i} \in A_{k}} a_{i} y_{i},
\ee
\par (ii) the difference
\be
\sum_{i=1}^{h+1} a_{i} x_{i} - \sum_{i=1}^{h+1} a_{i} y_{i}
\ee
is much larger in absolute value than any integer represented by
$B_{k-1}$,
\par (iii) \be
\sum_{x_{i} \in A_{k}} a_{i} x_{i} - \sum_{y_{i} \in A_{k}}
a_{i} y_{i} \equiv_{\mathscr{L}} \pm \left( \sum_{i=1}^{h} a_{i} d_{k,i}
+ a_{h+1} e_{k} \right).
\ee }
{\sc Claim 2} : {\em Suppose (iii) holds in Claim 1. Then
\be
\sum_{y_{i} \not\in A_{k}} a_{i} y_{i} -
\sum_{x_{i} \not\in A_{k}} a_{i} x_{i} \equiv_{\mathscr{L}}
\pm \sum_{i=1}^{h+1} a_{i} z_{i},
\ee
for some $(h+1)$-tuple $(z_{1},...,z_{h+1})$ of integers in
$B_{k-1}$.}
\\
\\
Indeed, suppose that $(x_{1},...,x_{h+1})$ and $(y_{1},...,y_{h+1})$ are any
two $(n+1)$-tuples in $B_{k}$ such that
\be
\sum_{i=1}^{n+1} a_{i} x_{i} = \sum_{i=1}^{n+1} a_{i} y_{i} = T, \;
{\hbox{say}}.
\ee
Then either (i) or (iii) in Claim 1 holds. But if (iii) holds then Claim 2
gives the contradiction that the integer $t_{k}$ is already represented by
$B_{k-1}$. Suppose (i) holds. Let $z$ be any element of $B_{k-1}$ and put
\be
x^{\prime}_{i} := \left\{ \begin{array}{lr} z, & {\hbox{if $x_{i} \in
A_{k}$}}, \\ x_{i}, & {\hbox{if $x_{i} \in B_{k-1}$}}, \end{array} \right.
\;\;\;
y^{\prime}_{i} := \left\{ \begin{array}{lr} z, & {\hbox{if $y_{i} \in
A_{k}$}}, \\ y_{i}, & {\hbox{if $y_{i} \in B_{k-1}$}}. \end{array} \right.
\ee
Then, since $B_{k-1}$ represents every integer at most once, we must have
that $(x^{\prime}_{1},...,x^{\prime}_{h+1})$ and
$(y^{\prime}_{1},...,y^{\prime}_{h+1})$ are equivalent representations of
$T$.
But then $(x_{1},...,x_{h+1})$ and $(y_{1},...,y_{h+1})$ are also
equivalent representations of $T$, so $A_{k}$ satisfies {\bf (II)} in this
case also.
\\
\\
{\sc Proof of Claim 1} : To simplify notation, put
\be
w_{i} := d_{k,i} \;\; {\hbox{for $i = 1,...,h$}}; \;\;\;\; w_{h+1} := e_{k}.
\ee
Consider the difference
\be
\sum_{x_{i} \in A_{k}} a_{i} x_{i} - \sum_{y_{i} \in A_{k}} a_{i} y_{i}
:= \sum_{i=1}^{h+1} c_{i} w_{i},
\ee
where
\be
c_{i} := \sum_{x_{u} = w_{i}} a_{u} - \sum_{y_{v} = w_{i}} a_{v}.
\ee
Alternative (i) trivially holds if all $c_{i} = 0$, so
so we may assume that some $c_{i} \neq 0$.
\par First suppose $c_{h+1} = 0$ and let $j \in [1,h]$ be the largest
index for
which $c_{j} \neq 0$. Then, for $M \gg 0$,
it is clear that the left-hand side of (2.18) is $\Theta(d_{k,j})$ and
hence alternative (ii) holds.
\par So finally we may suppose that $c_{h+1} \neq 0$. Let $f \in
\mathbb{Q}$ be such that $c_{h+1} = f \cdot a_{h+1}$. Then
\be
c_{h+1} w_{h+1} = - f a_{h} w_{h} + \Psi_{h},
\ee
where, for $M \gg 0$, the $\lq$error term' $\Psi_{h}$ must be much smaller
in absolute value than the $\lq$leading term' $- fa_{h} w_{h}$. Thus
alternative (ii) will hold unless $c_{h} = fa_{h}$. But then
\be
c_{h} w_{h} + c_{h+1}w_{h+1} = - fa_{h-1} w_{h-1} + \Psi_{h-1},
\ee
where, once again, for $M \gg 0$, the term $\Psi_{h-1}$ must be much
smaller in absolute value than $fa_{h-1} w_{h-1}$. Hence alternative (ii)
holds unless $c_{h-1} = fa_{h-1}$ and, by iteration of the same argument,
unless $c_{i} = fa_{i}$ for $i = 1,...,h$. In that case we thus have that
\be
\sum_{x_{i} \in A_{k}} a_{i} x_{i} -
\sum_{y_{i} \in A_{k}} a_{i} y_{i} \; \equiv_{\mathscr{L}} \;\; 
f \left( \sum_{i=1}^{n+1} a_{i} w_{i} \right).
\ee
But $f \cdot a_{i} \in \mathbb{Z}$
for $i = 1,...,h+1$ and since $\mathscr{L}$ is primitive, this implies that
$f \in \mathbb{Z}$. But then it is clear that we must have $|f| = 1$
and hence that alternative (iii) holds.
\par This completes the proof of Claim 1.
\\
\\
{\sc Proof of Claim 2} :
Without loss of generality we may assume that
\be
\sum_{x_{i} \in A_{k}} a_{i} x_{i} - \sum_{y_{i} \in A_{k}} a_{i} y_{i}
\equiv_{\mathscr{L}} \sum_{i=1}^{h+1} a_{i} w_{i},
\ee
and now need to construct an $(h+1)$-tuple $(z_{1},...,z_{h+1})$ of
integers in $B_{k-1}$ such that
\be
\sum_{y_{i} \in B_{k-1}} a_{i} y_{i} - \sum_{x_{i} \in B_{k-1}} a_{i} x_{i}
\equiv_{\mathscr{L}} \sum_{i=1}^{h+1} a_{i} z_{i}.
\ee
Let $i_{1} < i_{2} < \cdots < i_{m}$ be the indices for which $x_{i} \in
B_{k-1}$. We shall decompose the index set $\{1,...,h+1\}$ as the
disjoint union of $m+1$ subsets $S_{1},...,S_{m+1}$ defined as follows :
\\
\\
Fix $l$ with $1 \leq l \leq m$. Set $S_{l,0} := \{i_{l}\}$. For each
$j > 0$ set
\be
S_{l,j} := \{ i : x_{i} \in \{w_{k}\}_{k \in S_{l,j-1}} \}.
\ee
Noting that the sets $S_{l,j}$ are pairwise
disjoint for different $j$ and hence empty
for all $j \gg 0$, we set
\be
S_{l} := \bigsqcup_{j} S_{l,j}.
\ee
It is also easy to see that the sets $S_{1},...,S_{m}$ are pairwise disjoint.
We define
\be
S_{m+1} := \{1,...,h+1\} \backslash \bigsqcup_{l=1}^{m} S_{l}.
\ee
Note further that the sets $W_{1},...,W_{m+1}$ are pairwise disjoint, where
\be
W_{l} := \{ w_{i} : w_{i} = {\hbox{$x_{j}$ or $w_{j}$ for some $j \in
S_{l}$}}
\}, \;\;\;\; l = 1,...,m+1,
\ee
and that
\be
A_{k} = \{w_{1},...,w_{h+1}\} = \bigsqcup_{l=1}^{m+1} W_{l}.
\ee
Let $z$ be any element of $B_{k-1}$. We are now ready to define the
$(h+1)$-tuple $(z_{1},...,z_{h+1})$. Let $1 \leq i \leq h+1$. We put
\be
z_{i} := \left\{ \begin{array}{lr} y_{i}, & {\hbox{if $y_{i} \in
B_{k-1}$}}, \\
z, & {\hbox{if $y_{i} \in W_{m+1}$}}, \\ x_{i_{l}}, &
{\hbox{if $y_{i} \in W_{l}$, $1 \leq l \leq m$}}. \end{array} \right.
\ee
Then (2.24) now follows from (2.23), since the latter implies that
\be
\sum_{i \in S_{m+1}} a_{i} w_{i} \equiv_{\mathscr{L}} \sum_{i \in S_{m+1}}
a_{i} x_{i},
\ee
and, for $1 \leq l \leq m$, that
\be
\sum_{y_{i} \in W_{l}} a_{i} y_{i} \equiv_{\mathscr{L}}
\sum_{i \in S_{l}} a_{i} x_{i} - \sum_{i \in S_{l}} a_{i} w_{i} -
a_{i_{l}} x_{i_{l}}.
\ee
Hence the proof of Claim 2 is complete, and with it the proof of the first 
part of Theorem \ref{thm:unique}. It then follows immediately from the 
definitions (2.7), (2.8) and
(2.10) that, if the coefficients $a_{i}$ don't all have the same sign, then the
elements of $A$ in (2.4) can all be chosen to lie in any given half-line.
\end{proof}

\begin{rek}
If we were instead to work with ordered representations, then it is an 
immediate corollary of Theorem \ref{thm:unique} that there exists a 
unique representation basis for the form $a_1 x_1 + \cdots + a_h x_h$ if and
only if there do not exist two distinct subsets $I,I^{\prime}$ 
of $\{1,...,h\}$ such that 
\be
\sum_{i \in I} a_i = \sum_{i \in I^{\prime}} a_i.
\ee
This resolves Problem 16 in \cite{Na4}. 
\end{rek}

\begin{rek}\label{rek:spec}
The proof of Theorem \ref{thm:unique} simplifies considerably if the 
form $\mathscr{L}$ is partition regular, thus in particular in the case
of the forms $x_1 + \cdots + x_h$. See 
Remark \ref{rek:simp} below for an explanation. 
This is why we think our presentation streamlines those in earlier papers.  
\end{rek}


\setcounter{equation}{0}
\section{Partition regular forms}\label{sec:autos}

We wish to use the method of the previous section in order to generalise 
Theorem
\ref{thm:zbases}. As in 
\cite{Na4}, we adopt the following notations :
\be
\mathscr{F}_{0}(\mathbb{Z}) := \{f : \mathbb{Z} \rightarrow \mathbb{N}_0 \cup
\{\infty\} : 
f^{-1}(0) \; {\hbox{is finite}} \}, 
\ee
\be
\mathscr{F}_{\infty}(\mathbb{Z}) := 
\{f : \mathbb{Z} \rightarrow \mathbb{N}_0 \cup
\{\infty\} : 
f^{-1}(0) \; {\hbox{has asymptotic density zero}} \}.
\ee
Theorem \ref{thm:zbases} is a statement about $\mathscr{F}_{0}(\mathbb{Z})$. 
Only
minor modifications
to the method of Section 2 will be required, both to obtain a similar result 
for
general linear forms, and 
to extend the result to $\mathscr{F}_{\infty}(\mathbb{Z})$. We prepare the 
ground for this 
with a couple of lemmas. First some terminology :

\begin{defi}\label{defi:auto}
An {\em automorphism} of the linear form $\mathscr{L} : 
a_1 x_1 + \cdots + a_h x_h$
is a pair of 
functions $(\psi,\chi)$ from the set $\{x_1,...,x_h\}$ of variables to
$\{x_1,...,x_h\} \cup \{0\}$ such
that the linear form 
\be
\sum_{i=1}^{h} a_i \psi(x_i) - \sum_{i=1}^{h} a_i \chi(x_i)
\ee
is the same form as $\mathscr{L}$. The automorphism is said to be 
{\em trivial} if
$\chi \equiv 0$.
\end{defi}

Our first lemma is for the purpose of generalising Theorem \ref{thm:zbases} 
to other linear \\ forms :

\begin{lem}\label{lem:regular}
A linear form $\mathscr{L}$ is partition regular if and only if it possesses 
no non-trivial automorphisms.
\end{lem}

\begin{proof}
Denote $\mathscr{L} : a_1 x_1 + \cdots + a_h x_h$ as usual. First suppose
$\mathscr{L}$ is partition regular and thus, without loss of generality, that
$h \geq 2$ and $a_1 + \cdots + a_r = 0$ for some $2 \leq r \leq h$. Set 
\be
\psi(x_1) = 0, \;\;\; \psi(x_i) = x_i, \; i = 2,...,h, 
\ee
and 
\be
\chi(x_1) = 0, \;\;\; \chi(x_i) = x_1, \; i = 2,...,r, \;\;\; \chi(x_i) = 0, 
\; i = r+1,...,h.
\ee
Then one easily verifies that $(\psi,\chi)$ is a non-trivial automorphism 
of $\mathscr{L}$. 
\par Conversely, let $(\psi,\chi)$ be a non-trivial automorphism of 
$\mathscr{L}$. For each $i = 1,...,h$, (3.3) yields an equation 
between coeffcients of the form 
\be 
a_i = \sum_{j \in X_i} a_j - \sum_{j \in Y_i} a_j,
\ee
where $X_i$ and $Y_i$ are subsets of $\{1,...,h\}$. The definition of 
automorphism means that every index $j \in \{1,...,h\}$ occurs in 
at most one of the $X_i$ and at most one of the $Y_i$. Non-triviality
means that there is at least one $i$ such that $(X_i,Y_i) \neq 
(\{i\},\phi)$. Without loss of generality, suppose that
$(X_i,Y_i) \neq (\{i\},\phi)$ for $i = 1,...,r$ only, and some $r \leq h$. 
Adding together the left and right hand sides of (3.6) for $i = 1,...,r$
yields an equation of the form 
\be
a_1 + \cdots + a_r = \sum_{j \in X} a_j - \sum_{j \in Y} a_j,
\ee
for some disjoint subsets $X$ and $Y$ of $\{1,...,h\}$, with 
$X \subseteq \{1,...,r\}$. From (3.7) we can
extract a non-empty subset of the coefficients summing to zero, 
except if $X = \{1,...,r\}$ and $Y = \phi$. But it is easily seen that the
latter is impossible when $\chi \not\equiv 0$.  
\end{proof}

The next lemma is for the purpose of extending our results to
$\mathscr{F}_{\infty}(\mathbb{Z})$ : 

\begin{lem}\label{lem:density}
Let $S \subseteq \mathbb{Z}$ such that $d(S) = 1$. Let $l,m,p$ 
be any three positive integers. For each $n \in \mathbb{Z}$ set
\be
X_{l,m,p}(n) := \mathbb{Z} \cap 
\{\frac{a}{b} n+ c : (a,b,c) \in \mathbb{Z}^3 \cap \left( [-l,l] 
\times \pm[1,m] \times [-p,p] \right) \},
\ee
and set 
\be
S_{l,m,p} := \{n : X_{l,m,p}(n) \subseteq S \}.
\ee
Then $d(S_{l,m,p}) = 1$.
\end{lem}

\begin{proof}
The proof follows immediately from the following two facts :
\par (i) the intersection of finitely many sets of asymptotic 
density one has the same property
\par (ii) since $S$ has density one, the same is true, 
for any fixed integers $a,b,c$, with $b \neq 0$, of the set 
\be
\{n \in \mathbb{Z} : \frac{a}{b} n+ c 
\in S \cup (\mathbb{Q} \backslash \mathbb{Z}) \}
\ee 
\end{proof}

We are now ready to state the main result of this section :

\begin{thm}\label{thm:density}
Let $\mathscr{L} : a_1x_1 + \cdots + a_h x_h$ be any partition regular linear 
form. Then for any $f \in \mathscr{F}_{\infty}(\mathbb{Z})$, there 
exists a subset $A \subseteq \mathbb{Z}$ such that $f_{A,\mathscr{L}} = f$.
\end{thm}

\begin{proof}
Let a partition regular $\mathscr{L} : a_1 x_1 + \cdots + a_h x_h$ and 
$f \in \mathscr{F}_{\infty}(\mathbb{Z})$ be given : 
we shall show how to construct 
$A \subseteq \mathbb{Z}$ with $f_{A,\mathscr{L}} = f$. Let $\mathscr{M}$ be 
the multisubset of $\mathbb{Z}$ consisting of $f(n)$ repititions of $n$ for 
every $n$. The problem amounts to constructing a $\lq$unique representation 
basis' for $\mathscr{M}$. This is some 
countable set : let $\mathscr{O} = \{\tau_{1},\tau_{2},...\}$ be any 
well-ordering of it. 
We now construct $A$ step-by-step as in (2.4)-(2.10). This time the ordering
$\mathscr{O}$ is as just defined above. We'll have $t_k = \tau_{k^{\prime}}$
for some $k^{\prime}$ depending on $k$. Two requirements must be satisfied 
when we tag on the numbers $d_{k,1},...,d_{k,h},e_k$ to our set $A$ :
\\
\\
{\bf I.} No new representation is created of any number appearing 
before $t_k$ in the ordering $\mathscr{O}$. 
\\
{\bf II.} No representation is created of any integer $n$ for which 
$f(n) = 0$. 
\\
\\
To satisfy these requirements, the integer denoted $M$ in (2.7) will now have 
to depend on $k$. To begin with this is because, when considering {\bf I}, 
since the ordering $\mathscr{O}$ is chosen randomly, we have no control over 
how quickly the sizes of numbers in this ordering grow as ordinary integers. 
Clearly, $M = M_k$ can be chosen large enough to take account of this 
difficulty in the sense that Claim 1 holds as before. 
We still need to rule out case (iii) of that claim occurring, and it is here 
that we make use of the assumption that $\mathscr{L}$ is partition regular, 
for (2.13) describes an automorphism of $\mathscr{L}$, just as long as the 
numbers $d_{k,1},...,d_{k,h},e_k$ are distinct. 
\par The only remaining problem is thus {\bf II}. 
Let $S := \mathbb{Z} \backslash f^{-1}(0)$. By assumption
$d(S) = 1$. Then it follows from Lemma \ref{lem:density} that 
there exists a choice of a sufficiently large $M = M_k$ which will mean that 
{\bf II} is indeed satisfied.
Indeed once $d_{k,1},...,d_{k,h-1}$ have been chosen with due 
regard to {\bf I}, one just  
needs to choose $d_{k,h}$ to also lie in a set $S_{l,m,p}$, where 
$l,m,p$ are fixed integers depending a priori on all of the numbers 
$d_{k,1},...,d_{k,h-1},t_k,a_1,...,a_{h+1},s_1,...,s_{h+1}$. 
\par Thus there is indeed a choice of $M_k$ that works at each step, and the 
theorem is proved.
\end{proof}

\begin{rek}\label{rek:simp}
The proof of Theorem \ref{thm:unique} simplifies for partition regular 
forms in the same way as in the argument just presented. Namely, we can 
ignore Case (iii) of \\ Claim 1, and thus don't need the most technical
part of the proof, which is the proof of Claim 2. 
\end{rek}   
  
\setcounter{equation}{0}
\section{The form $x_1 - x_2$}\label{sec:diff}

We do not know if there exist any partition irregular forms for which 
Theorem \ref{thm:density} still holds. For the simplest such form, namely
$x_1 - x_2$, this is clearly not the case. Henceforth we denote this form
by $\mathscr{L}_{0}$. Obviously any function $f$ represented by $\mathscr{L}_0$
must be even and satisfy $f(0) = 1$. There is a more serious
obstruction, however. Let $f : \mathbb{Z} \rightarrow \mathbb{N}_0 \cup 
\{\infty\}$ and suppose $f(n) \geq 3$ for some $n$. Suppose 
$f_{A,\mathscr{L}_0} = f$ for some $A \subseteq \mathbb{Z}$ and let 
$a_1,...,a_6 \in A$ be such that 
\be
a_1 - a_2 = a_3 - a_4 = a_5 - a_6 = n,
\ee
are three pairwise non-equivalent representations of $n$ (i.e.: the 
numbers $a_1,a_3,a_5$ are distinct). Then we also have the equalities
\be
a_1 - a_3 = a_2 - a_4, \;\;\; a_3 - a_5 = a_4 - a_6, \;\;\;
a_1 - a_5 = a_2 - a_6,
\ee
and at least one of these three differences must be different from $n$. Thus
there exists some other number $m$ for which $f(m) \geq 2$. 
\par The following definition captures this kind of condition imposed on
a function $f$ representable by $\mathscr{L}_0$ :
 
\begin{defi}\label{defi:plentiful}
Let $f : \mathbb{Z} \rightarrow \mathbb{N}_0 \cup \{\infty\}$. 
A sequence (finite or infinite)
$s_{1},s_{2},s_{3},...$ of positive integers is said to be
{\em plentiful} for
$f$ if, for every pair $l \leq m$ of positive integers, we have
\be
f \left( \sum_{i=l}^{m} s_{i} \right) > 1.
\ee
\end{defi}

The main result of this section is the following :

\begin{thm}\label{thm:xminusy} 
Let $f \in \mathscr{F}_{0}(\mathbb{Z})$ 
be even with $f(0) = 1$. 
\par (i) If $f^{-1}(\infty) \neq \phi$ then $f$ is representable by 
$\mathscr{L}_0$ if and only if there exists an infinite 
plentiful sequence for $f$.
\par (ii) If $f^{-1}(\infty) = \phi$ but $f$ is unbounded, then $f$ is 
representable by $\mathscr{L}_0$ if and only if there exist arbitrarily 
long plentiful sequences for $f$.
\end{thm}

We do not know whether this result can be extended to functions in 
$\mathscr{F}_{\infty}(\mathbb{Z})$, nor exactly which bounded functions
in $\mathscr{F}_{0}(\mathbb{Z})$ can be represented by $\mathscr{L}_0$. 

\begin{proof}
Throughout this proof, since we are working with a fixed form 
$\mathscr{L}_0$, we will write simply $f_A$ for the representation function 
of a subset $A$ of $\mathbb{Z}$. 
\\
\\
We begin with the proof of part (i) : that for part (ii) will be similar.
First suppose $f$ is representable by $\mathscr{L}_0$ and let
$A \subseteq \mathbb{Z}$ be such that $f_A = f$.
Suppose $f(n) = \infty$. Let
$(x_{i},y_{i})_{i=1}^{\infty}$ be a sequence of pairs of elements of $A$
such that $x_{i} - y_{i} = n$
for each $i$ and such that the sequence $(x_{i})$ is either strictly
increasing or strictly decreasing. Let
\be
s_{i} := \left\{ \begin{array}{lr} x_{i+1} - x_{i}, & {\hbox{if $(x_{i})$
increasing}}, \\ x_{i} - x_{i+1}, & {\hbox{if $(x_{i})$ decreasing}}.
\end{array} \right.
\ee
Let $1 \leq l \leq m$. Then
\be
x_{m+1} - y_{m+1} = x_{l} - y_{l} \Rightarrow x_{m+1} - x_{l} =
y_{m+1} - y_{l} = \pm \sum_{i=l}^{m} s_{i},
\ee
and hence
\be
f_{A} \left( \sum_{i=l}^{m} s_{j} \right) \geq 2.
\ee
Thus the sequence $(s_{i})$ is plentiful for $f$.
\\
\\
Conversely, suppose there exists an infinite plentiful sequence
$(s_i)_{i = 1}^{\infty}$ for $f$. We will construct a set $A$ which 
represents $f$. Set
\be
S := \left\{ \sigma_{l,m} := \sum_{i=l}^{m} s_{i} \; \mid \;
1 \leq l \leq m < \infty \right\}.
\ee
Set
\be
\mathscr{I} := \{n : f(n) > 1\}, \;\;\; \mathscr{J} := \{n : f(n) = 1 \},
\ee
and note that $S \subseteq \mathscr{I}$. Let $\mathscr{M}$ be the multisubset
of $\mathbb{Z}$ consisting of $f(n)$ copies of $n$ for each $n$, and 
$\mathscr{O}$ any well-ordering of $\mathscr{M}$. Put 
$A_0 := \{d_0\}$ for any choice of a non-zero integer $d_0$. The set 
$A$ will be constructed step-by-step as in (2.4). At each step $k > 0$ the
set $A_k$ will consist of two suitably chosen integers $x_k$ and $y_k$. 
Our choices will be made so as to ensure that the 
following two requirements are satisfied for every $k$ : 
\\
\\
{\bf (I)} If
\be
\mathscr{U}_{k} := \{ n : f_{B_{k}}(n) > f_{B_{k-1}}(n) \},
\ee
then $t_k \in \mathscr{U}_{k}$, where $t_k$ is the least number in the
ordering $\mathscr{O}$ not yet represented by $B_{k-1} := \sqcup_{j=0}^{k-1}
A_j$. Moreoever, for any $n \in \mathscr{U}_{k}$, then
\be
f_{B_{k-1}}(n) < f(n),
\ee
\be
f_{B_{k}}(n) \leq f_{B_{k-1}}(n) + 2,
\ee
and
\be
{\hbox{if}} \; f_{B_{k}}(n) = f_{B_{k-1}}(n) + 2 \; {\hbox{then}} \;
f_{B_{k-1}}(n) = 0 \; {\hbox{and}} \; n \in S.
\ee
{\bf (II)} Suppose $n \in \mathscr{I}$. Let $f_{B_{k}}(n) := p_{k}$
and
\be
a_{1} - b_{1} = \cdots = a_{p_{k}} - b_{p_{k}} = n,
\ee
be the different representations of $n$ in $B_{k}$, where $a_{1} < a_{2} <
\cdots < a_{p_{k}}$. Then there exist integers $0 < m_{1} < m_{2} <
\cdots < m_{p_{k}}$ such that
\be
a_{i+1} - a_{i} = \sigma_{m_{i+1},m_{i}}, \;\;\;\;\;\; i = 1,...,p_{k} - 1.
\ee
It is clear that if {\bf (I)} and {\bf (II)} 
are satisfied for every $k \geq 0$,
then the set $A$ given by (2.4) represents $f$. The condition {\bf (II)}
will be
useful in establishing (4.12). The elements of the
different $A_k$ are chosen inductively.
Observe that {\bf (I), (II)} are trivially satisfied for $k = 0$, so 
suppose $k > 0$ and that {\bf (I), (II)} are satisfied for each $k^{\prime} < k$.
We now describe how the elements of $A_{k}$ may be chosen. Let
\be
M_{k} := \max \{ |n| : n \in B_{k-1} \}.
\ee
and note that $f_{B_{k-1}}(n) = 0$ for all $n > 2M_{k}$.
\\
\\
{\sc Case I} : $t_k \in \mathscr{J}$.
\\
\\
Then {\bf (II)} will continue to hold no matter what we do. 
We now choose $x_{k}$ to
be any integer greater than $2|t_{k}| + 3M_{k}$, and choose $y_{k} :=
x_{k} - t_{k}$. This choice of $x_{k}$ and $y_{k}$ guarantees that, if
$n \in \mathscr{U}_{k}$, then
\par (a) $f_{B_{k-1}}(n) = 0$ and
\par (b) $f_{B_{k}}(n) = 1$,
\\
hence that {\bf (I)} is satisfied. To verify (a), we observe that if $n \in
\mathscr{U}_{k}$ then either $n = \pm t_{k} = \pm (x_{k} - y_{k})$ or $|n|
> 2M_{k}$. For (b), we note that if $a,b \in B_{k-1}$ and $x_{k} - a =
y_{k} - b$, then $x_{k} - y_{k} = a-b$, contradicting the assumption
that $t_k \in \mathscr{J}$.
\\
\\
{\sc Case II} : $t_k \in \mathscr{I}$.
\\
\\
Let
$p := f_{B_{k-1}}(t_k)$. If $p = 0$ then proceed as in Case I. Otherwise 
let $(a_{i},b_{i})_{i=1}^{p}$ be the
different representations of $t_k$ in $B_{k-1}$ and let $m_{1},...,m_{p}$ be
the
integers for which (4.13) is satisfied (with $k-1$ instead of $k$). We choose
\be
x_{k} := a_{p} + \sigma_{m_{p+1},m_{p}},
\ee
for some sufficiently large integer $m_{p+1}$ such that $x_{k} >
2|t_k| + 3M_{k}$. Then we take $y_{k} := x_{k} - t_k$.
Reasoning as in Case I, the size of $x_{k}$ and $y_{k}$ guarantee that (4.10)
and (4.11) will be satisfied, and the relationships (4.14) and (4.16) 
will imply (4.12), and thus ensure that {\bf (II)} still holds.
\par This completes the induction step, and hence the proof of part (i) of
the theorem.
\\
\\
Now we briefly outline the proof of part (ii). That representability of 
$f$ implies the existence of arbitrarily long plentiful sequences is 
shown in the same way as before. Suppose now such sequences exist. The 
construction of $A$ such that $f_A = f$ proceeds as above and, with notation
as before, the only difference is 
in the inductive choice of the elements in $A_k$ for $k > 0$. Suppose we
have already chosen $A_k^{\prime}$ for $k^{\prime} < k$ so that 
\be
f_{B_{k-1}}(n) \leq f(n) \;\;\;\; \forall \; n \in \mathbb{Z}.
\ee
Let $t_k$ be the least integer in the ordering $\mathscr{O}$ for which 
$f_{B_{k-1}}(t_k) < f(t_k)$. Set 
$\gamma_{k} := f(t_{k}) - f_{B_{k-1}}(t_{k})$. The set $A_{k}$ will be the 
union of $2\gamma_{k}$ elements $x_{i,k}, y_{i,k}$, $i = 1,...,\gamma_{k}$. 
\\
\\
First, with $M_k$ defined as in (4.15), 
we choose $x_{1,k}$ to be any integer greater than $2|t_{k}| + 
3M_{k}$ and take $y_{1,k} := x_{1,k} - t_{k}$. Since there exist arbitrarily long
plentiful sequences, we 
can find such a sequence
\be
0 < a_{1} < a_{2} < \cdots < a_{\gamma_{k}-1}
\ee
such that each of the quotients 
\be
{a_{1} \over x_{1,k}}, \;\;\; {a_{i+1} \over a_{i}}, \;\;\; i = 1,...,
\gamma_{k} - 1
\ee
is arbitrarily large. We then wish to choose the remaining elements of 
$A_{k}$ as 
\be
x_{i,k} := x_{i-1,k} + a_{i-1}, \;\;\;\;\;\; 
y_{i,k} := x_{i,k} - t_{k}, \;\;\;\;\;\; i = 2,...,\gamma_{k}.
\ee
Provided the quotients in (4.19) are all sufficiently large, it is clear that,
if $n \in \mathscr{U}_{k}$, then either 
\par (a) $n = \pm t_{k}$ and $f_{B_{k}}(n) = f_{B_{k-1}}(n) + \gamma_{k} = 
f(n)$, or
\par (b) $f_{B_{k-1}}(n) = 0$ and either $f_{B_{k}}(n) = 1$, or 
$f_{B_{k}}(n) = 2$ and $n \in \mathscr{I}$. 
\\
Thus (a) and (b) ensure that (4.17) is also satisfied for this value of $k$, 
and thus the set $A$ given by (2.4) will satisfy $f_A = f$. 
\par Hence the proof of Theorem \ref{thm:xminusy} is complete.    
\end{proof}

\setcounter{equation}{0}
\section{Open problems}\label{sec:conclusions}

We only mention what are probably the two most glaring issues left unresolved
by the investigations above. 
\\
\\
1. Which functions $f \in \mathscr{F}_{\infty}(\mathbb{Z})$ are representable 
by a partition irregular form ? Does there exist such a form which
represents any such function ? For the form $\mathscr{L}_{0}$, we want to 
know if Theorem \ref{thm:xminusy} can be extended to 
$\mathscr{F}_{\infty}(\mathbb{Z})$, 
and which bounded functions are representable. 
\\
\\
2. The methods employed in this paper to construct sets with 
given representation functions have, in common with previous similar methods, 
the obvious weakness that they produce very sparse sets. It is an 
important unsolved problem to find the maximal possible density of 
a set with a given representation function : this problem is still 
unsolved for every possible $f$ and $\mathscr{L}$, though the most
natural case to look at is $f \equiv 1$. It should be investigated to what
extent existing optimal constructions for the forms $x_1 + \cdots + x_h$ can
be extended to general forms. 
  
\section*{Acknowledgement}

I thank Boris Bukh for a very helpful discussion.

\ \\

\end{document}